\documentclass[12pt]{article}
\usepackage[margin=2.8cm]{geometry}
\usepackage[all]{xy}
\usepackage{amsmath,amsfonts,amssymb,amsthm,epsfig,epstopdf,titling,url,array}
\usepackage[colorlinks=true]{hyperref}
\usepackage{tikz}
\theoremstyle{plain}
\newtheorem{thm}{Theorem}[subsection]

\newtheorem{prop}[thm]{Proposition}
\newtheorem{cor}{Corollary}[subsection]
\theoremstyle{definition}
\newtheorem{defn}{Definition}[subsection]

\theoremstyle{remark}
\newtheorem{remark}{Remark}[subsection]

\theoremstyle{proof}

\theoremstyle{claim}

\author{SELCAN AKSOY\\ Florida State University\\ Faculty of Mathematics\\ E-mail: saksoy@math.fsu.edu}
\title{Composition of Roofs in Derived Category}
\begin{document}
\maketitle
\begin{center}
\textsc{$\mathbf{Abstract}$}
\end{center}
In that paper, we prove that the composition of two roofs is another roof by using mapping cone of a morphism of cochain complexes. 
\tableofcontents
\section{Introduction}
Assume that $\mathcal{A}$ is an abelian category. P. Aluffi defines a mapping cone $MC(f)$ of a morphism $f:~A\rightarrow B$ in $C(\mathcal{A})$ and homotopy between two morphisms in that category in {\color{blue} \cite{al}}. Also, in {\color{blue} \cite{kr}}, H. Krause defines triangulated category and the localizing class. After that, he proves that the homotopic category $K(\mathcal{A})$ is triangulated in {\color{red} Section 2.5}.\\

Using this information, we prove that for a given upside down roof in the localization of $K(\mathcal{A})$, we can obtain a regular roof in that category. This allows us to compute the composition of two regular roofs in the localization of homotopic category.

\section{Mapping Cone and Homotopy}
The collection $C(\mathcal{A})$ consisting of all cochain complexes in an abelian category $\mathcal{A}$ forms an abelian category. It is easy to show that the set of morphisms of that category is an abelian group, finite products and coproducts exist since they exist in $\mathcal{A}$.
\begin{defn}
A morphism $f$ between cochain complexes is quasi isomorphism if it induces an isomorphism in cohomology.
\end{defn}

For a given morphism $f:~A\rightarrow B$ between cochain complexes $A$ and $B$, we define a mapping cone $MC(f)$ as $MC(f)^i=A[1]^i\oplus B^i=A^{i+1}\oplus B^i$ for all $i$. Here, we get the morphisms 
\begin{align*}
& d^i_{MC(f)}:~MC(f)^i\rightarrow MC(f)^{i+1},\\
& d^i_{MC(f)}(a,~b)=(-d^{i+1}_A(a),~f^{i+1}(a)+d^i_B(b))
\end{align*}

between those objects.\\

$MC(f)$ is a cochain complex since $d^{i+1}_{MC(f)}\circ d^i_{MC(f)}=0$.
\begin{align*}
\begin{tikzpicture}
\node (A) at (0, 0) {$...$};
\node (B) at (3, 0) {$A^{i+1}$};
\node (C) at (6, 0) {$A^{i+2}$};
\node (D) at (9, 0) {$A^{i+3}$};
\node (E) at (12, 0) {$...$};
\node (F) at (0, -2) {$...$};
\node (G) at (3, -2) {$B^i$};
\node (H) at (6, -2) {$B^{i+1}$};
\node (J) at (9, -2) {$B^{i+2}$};
\node (K) at (12, -2) {$...$};
\node (L) at (3, -1) {$\oplus$};
\node (M) at (6, -1) {$\oplus$};
\node (N) at (9, -1) {$\oplus$};
\path[->] (A) edge (B);
\path[->] (B) edge node [above] {$-d^{i+1}_A$} (C);
\path[->] (C) edge node [above] {$-d^{i+2}_A$} (D);
\path[->] (D) edge (E);
\path[->] (F) edge (G);
\path[->] (G) edge node [below] {$d^i_B$} (H);
\path[->] (H) edge node [below] {$d^{i+1}_B$} (J);
\path[->] (J) edge (K);
\path[->] (B) edge node [right, midway]{$f^{i+1}$} (H);
\path[->] (C) edge node [right, midway]{$f^{i+2}$} (J);
\end{tikzpicture}
\end{align*}
\begin{defn}
A homotopy $k$ between two morphisms of cochain complexes $f,~g:~A\rightarrow B$ is a collection of morphisms $k^i:~A^i\rightarrow B^{i-1}$ such that for all $i$, 
\begin{align*}
g^i -f^i=d^{i-1}_B\circ k^i + k^{i+1}\circ d^i_A.
\end{align*}
\end{defn}

The above morphisms $f$ and $g$ are homotopic if there is a homotopy between them. We use the following diagram to show that homotopy and use the symbol $f\sim g$ to mean there exists a homotopy between the morphisms $f$ and $g$.
\begin{align*}
\begin{tikzpicture}
\node (A) at (0, 0) {$...$};
\node (B) at (3, 0) {$A^{i-1}$};
\node (C) at (6, 0) {$A^i$};
\node (D) at (9, 0) {$A^{i+1}$};
\node (E) at (12, 0) {$...$};
\node (F) at (-3, -2) {$...$};
\node (G) at (0, -2) {$B^{i-2}$};
\node (H) at (3, -2) {$B^{i-1}$};
\node (J) at (6, -2) {$B^i$};
\node (K) at (9, -2) {$...$};
\path[->] (A) edge (B);
\path[->] (B) edge node [above] {$d^{i-1}_A$} (C);
\path[->] (C) edge node [above] {$d^i_A$} (D);
\path[->] (D) edge (E);
\path[->] (F) edge (G);
\path[->] (G) edge node [below] {$d^{i-2}_B$} (H);
\path[->] (H) edge node [below] {$d^{i-1}_B$} (J);
\path[->] (J) edge (K);
\path[->] (B) edge node [left, midway]{$k^{i-1}$} (G);
\path[->] (C) edge node [left, midway]{$k^i$} (H);
\path[->] (D) edge node [left, midway]{$k^{i+1}$} (J);
\end{tikzpicture}
\end{align*}
\begin{defn}
A morphism $f:~A\rightarrow B$ is a homotopy equivalence if there is a morphism $g:~B\rightarrow A$ such that $f\circ g \sim id_B$ and $g\circ f\sim id_A$.
\end{defn}

$A$ and $B$ are homotopy equivalent if there is a homotopy equivalence $A\rightarrow B$.
\begin{prop}
{\color{blue} \cite{al}} If $f,~g:~A\rightarrow B$ are homotopic, then $H^{\bullet}(f)=H^{\bullet}(g)$.
\end{prop}
\begin{cor}
If $f:~A\rightarrow B$ is homotopy equivalence, then $H^{\bullet}(A)\cong H^{\bullet}(B)$.
\end{cor}

Every homotopy equivalence is a quasi isomorphism, but every quasi isomorphism may not be a homotopy equivalence. 

\section{Triangulated Categories}
\begin{defn}
{\color{blue} \cite{kr}} Assume that $\mathcal{A}$ is an additive category with an equivalence $\mathcal{F}:~\mathcal{A}\rightarrow \mathcal{A}$. A triangle $(f,~g,~h)$ in $\mathcal{A}$ is a sequence of morphisms $\xymatrix{ X\ar[r]^f & Y\ar[r]^g & Z\ar[r]^h & \mathcal{F}(X)}$ for all objects $X$, $Y$ and $Z$ in $\mathcal{A}$.\\

A morphism between two triangles $(f_1,~g_1,~h_1)$ and $(f_2,~g_2,~h_2)$ is a triple $(k_1,~k_2,~k_3)$ of morphisms in $\mathcal{A}$ making the following diagram commute.
\begin{equation*}
\xymatrix{ X\ar[d]_{k_1} \ar[r]^{f_1} & Y\ar[d]_{k_2} \ar[r]^{g_1} & Z\ar[r]^{h_1} \ar[d]_{k_3} & \mathcal{F}(X) \ar[d]_{\mathcal{F}(k_1)}\\ X'\ar[r]_{f_2} & Y'\ar[r]_{g_2} & Z'\ar[r]_{h_2} & \mathcal{F}(X')}
\end{equation*}

The category $\mathcal{A}$ is called pre-triangulated if it has a class of exact triangles satisfying the following conditions.
\begin{enumerate}
\item A triangle is exact if it is isomorphic to an exact triangle. 
\item For all objects $X$ in $\mathcal{A}$, the triangle $\xymatrix{ 0\ar[r] & X\ar[r]^{id} & X\ar[r] & 0}$ is exact.
\item Each morphism $f:~X\rightarrow Y$ can be completed to an exact triangle $(f,~g,~h)$.
\item A triangle $(f,~g,~h)$ is exact if and only if the triangle $(g,~h,~-\mathcal{F}(f))$ is exact.
\item Given two exact triangles $(f_1,~g_1,~h_1)$ and $(f_2,~g_2,~h_2)$, each pair of maps $k_1$ and $k_2$ satisfying $k_2\circ f_1=f_2\circ k_1$ can be completed to a morphism;
\begin{equation*}
\xymatrix{ X\ar[d]_{k_1} \ar[r]^{f_1} & Y\ar[d]_{k_2} \ar[r]^{g_1} & Z\ar[r]^{h_1} \ar[d]_{k_3} & \mathcal{F}(X) \ar[d]_{\mathcal{F}(k_1)}\\ X'\ar[r]_{f_2} & Y'\ar[r]_{g_2} & Z'\ar[r]_{h_2} & \mathcal{F}(X')}
\end{equation*}

$\mathcal{A}$ is a triangulated category if in addition it satisfies the following axiom.
\item {\bf The Octahedral Axiom:} Given exact triangles $(f_1,~f_2,~f_3)$, $(g_1,~g_2,~g_3)$ and $(h_1,~h_2,~h_3)$ with $h_1=g_1 \circ f_1$, there exists an exact triangle $(k_1,~k_2,~k_3)$ making the following diagram commutative.
\begin{equation*}
\xymatrix{ X\ar[d]_= \ar[r]^{f_1} & Y\ar[d]^{g_1} \ar[r]^{f_2} & U\ar[r]^{f_3} \ar[d]^{k_1} & \mathcal{F}(X) \ar[d]^=\\ X\ar[r]^{h_1} & Z\ar[d]^{g_2} \ar[r]^{h_2} & V\ar[d]^{k_2} \ar[r]^{h_3} & \mathcal{F}(X) \ar[d]^{\mathcal{F}(f_1)}\\ & W\ar[r]^= \ar[d]^{g_3} & W\ar[r]^{g_3} \ar[d]^{k_3} & \mathcal{F}(Y)\\ & \mathcal{F}(Y) \ar[r]^{\mathcal{F}(f_2)} & \mathcal{F}(U)}
\end{equation*}
\end{enumerate}
\end{defn}
\begin{remark}
If $\mathcal{A}$ is a pretriangulated category, then $A^{op}$ is a pretriangulated category, too.
\end{remark}

\section{The Localization of A Category}
\begin{defn}
{\color{blue} \cite{kr}} Assume that $\mathcal{A}$ is a category and $F$ is a class of maps in $\mathcal{A}$. $F$ is a localizing class if the following conditions are satisfied.
\begin{enumerate}
\item If $f,~g$ are composible maps in $F$, then $g\circ f$ is in $F$.
\item The identity map $id_A$ is in $F$ for all $A\in \mathcal{A}$.
\item If $f:~A\rightarrow B$ is in $F$, then every pair of maps $B'\rightarrow B$ and $A\rightarrow A''$ in $\mathcal{A}$ can be completed to a pair of commutative diagrams;
\begin{align*}
\xymatrix{ A'\ar[r] \ar[d]^{f'} & A \ar[d]^{f}\\ B'\ar[r] & B} \quad \xymatrix{ A\ar[r] \ar[d]^{f} & A''\ar[d]^{f''}\\ B\ar[r] & B''}  
\end{align*}

such that $f'$ and $f''$ are in $F$.
\item If $f,~g:~A\rightarrow B$ are maps in $\mathcal{A}$, then there is some $h:~A'\rightarrow A$ in $F$ with $f\circ h=g\circ h$ if and only if there is some $k:~B\rightarrow B''$ in $F$ with $k\circ f=k\circ g$.
\end{enumerate}
\end{defn}
\begin{defn}
{\color{blue} \cite{kr}} Assume that $\mathcal{A}$ is a category and $F$ is a class of maps in $\mathcal{A}$. The localization of $\mathcal{A}$ with respect to $F$ is a category $\mathcal{A}[F^{-1}]$ together with a functor $\mathcal{F}:~\mathcal{A} \rightarrow \mathcal{A}[F^{-1}]$ such that $\mathcal{F}(f)$ is an isomorphism for all $f$ in $F$ and any functor $\mathcal{G}:~\mathcal{A} \rightarrow \mathcal{B}$ such that $\mathcal{G}(f)$ is an isomorphism for all $f$ in $F$ factors uniquely through $\mathcal{F}$.
\end{defn}

We can always find a localization like that.
\begin{defn}
Assume that $\mathcal{A}$ is a category and $F$ is a localizing class. The objects of $\mathcal{A}[F^{-1}]$ are the objects of $\mathcal{A}$. The morphisms $A\rightarrow B$ in $\mathcal{A}[F^{-1}]$ are equivalence classes of diagrams $\xymatrix{A & \ar[l]_f B'\ar[r]^g & B}$ with the morphism $f$ in $F$ for all objects $A$ and $B$ in the category $\mathcal{A}[F^{-1}]$. We will call those morphisms as regular roofs.
\end{defn}

A pair $(f,~g)$ is also called a fraction because it is written as $g\circ f^{-1}$ in $\mathcal{A}[F^{-1}]$.\\

\begin{remark}
The functor $\mathcal{F}:~\mathcal{A}\rightarrow \mathcal{A}[F^{-1}]$ sends a map $f:~A\rightarrow B$ to the pair $(id_A,~f)$. 
\end{remark}
\begin{defn}
$(f,~g)$ and $(f',~g')$ are equivalent if there exists a commutative diagram with $f''$ in $F$;
\begin{equation*}
\xymatrix{ & B'\ar[ld]_f \ar[rd]^g\\ A & \ar[l]_{f''} B''' \ar[r]^{g''} \ar[u] \ar[d] & B\\ & B''\ar[ul]^{f'} \ar[ur]_{g'}}
\end{equation*}
\end{defn}

\section{Composition of Two Roofs}
\begin{defn}
{\color{blue} \cite{al}} Assume that $\mathcal{A}$ is an abelian category. $K(\mathcal{A})$ is a category whose objects are the objects in $C(\mathcal{A})$ and the set of morphisms is
\begin{align*}
Hom_{K(\mathcal{A})}(A,~B)=Hom_{C(\mathcal{A})}(A,~B) / \sim
\end{align*}

where $\sim$ is homotopy relation. 
\end{defn}

If $f\circ g\sim id$ in $C(\mathcal{A})$, then $f\circ g=id$ in $K(\mathcal{A})$. As a result, homotopy equivalences in $C(\mathcal{A})$ become isomorphisms in $K(\mathcal{A})$ and we say that $K(\mathcal{A})$ is obtained by inverting all homotopy equivalences in $C(\mathcal{A})$. It is an additive category, but not abelian in general since homotopic maps don't have same kernels and cokernels.\\

In {\color{blue} \cite{kr}}, H. Krause proves that $K(\mathcal{A})$ is a triangulated category.\\ 

\begin{remark}
The set of quasi isomorphisms in $K(\mathcal{A})$ for a given abelian category $\mathcal{A}$ forms a localizing class.
\end{remark}
\begin{thm}
{\color{blue} \cite{al}} Assume that $\mathcal{A}$ is an abelian category and we have two morphisms $\xymatrix{L\ar[r]^{\alpha} & \underline{K}}$ and $\xymatrix{M\ar[r]^{\beta} & \underline{K}}$ with $\beta$ is a quasi isomorphism for objects $L$ and $M$ in $K(\mathcal{A})$. Then, there exists a cochain complex $K$, morphisms $K\rightarrow L$ which is quasi isomorphism and $K\rightarrow M$ in $K(\mathcal{A})$ such that the following diagram commutes.
\begin{equation}
\label{1}
\xymatrix{ & K\ar[rd]^{\gamma_1} \ar[ld]_{\gamma_2}\\  L\ar[rd]_{\alpha} & & M\ar[ld]^{\beta}\\ & \underline{K}}
\end{equation}
\end{thm}
\begin{proof}
Assume that $\gamma$ is the composition $\xymatrix{L \ar[r] & \underline{K} \ar[r] & MC(\beta)}$, $K=MC(\gamma)[-1]$ and $K^i=L^i\oplus M^i\oplus \underline{K}^{i-1}$. We define morphisms $K^i\rightarrow L^i,~(l,~m,~\underline{k})\rightarrow l$ and $K^i\rightarrow M^i,~(l,~m,~\underline{k})\rightarrow -m$ as in {\color{blue} \cite{al}}.\\

We want to prove that $L$ and $M$ are connected by a regular roof as well.\\

For the rest of the proof, we need to show that the {\color{red} Diagram \ref{1}} commutes.\\

$H^{\bullet}(M)\cong H^{\bullet}(\underline{K})$ since $\beta$ is a quasi isomorphism. This implies that $MC(\beta)$ is exact, so $H^{\bullet}(MC(\beta))=0$.
\begin{align*}
MC(\beta)^i=M[1]^i\oplus \underline{K}^i=M^{i+1}\oplus \underline{K}^i,
\end{align*}
\begin{align*}
d^i_{MC(\beta)}:~MC(\beta)^i\rightarrow MC(\beta)^{i+1},~d^i_{MC(\beta)}(m,~\underline{k})=(-d^{i+1}_M,~\beta^{i+1}(m)+d^i_{\underline{K}}(\underline{k})).
\end{align*}

We define $\gamma^i(l)=(0,~\alpha^i(l))$ and 
\begin{align*}
MC(\gamma)^i=L[1]^i\oplus M^{i+1}\oplus \underline{K}^i=L^{i+1}\oplus M^{i+1}\oplus \underline{K}^i
\end{align*}

where $d^i_{MC(\gamma)}:~MC(\gamma)^i\rightarrow MC(\gamma)^{i+1}$ with
\begin{align*}
d^i_{MC(\gamma)}(l,~m,~\underline{k})=(-d_L^{i+1}(l),~\gamma^{i+1}(l)+d^i_{MC(\beta)}(m,~\underline{k}))=\\ 
(-d_L^{i+1}(l), -d_M^{i+1}(m),~\alpha^{i+1}(l)+\beta^{i+1}(m)+d_{\underline{K}}^i(\underline{k})).
\end{align*}
\begin{align*}
K=MC(\gamma)[-1],
\end{align*}
\begin{align*}
MC(\gamma)[-1]^i=MC(\gamma)^{i-1}=L^i\oplus M^i\oplus \underline{K}^{i-1}=K^i
\end{align*}

 and $d_K^i=-d_{MC(\gamma)}^{i-1}$ with 
\begin{align*}
d^i_K=(d_L^i(l),~d_M^i(m),~-\alpha^i(l)-\beta^i(m)-d^{i-1}{\underline{K}}(\underline{k})).
\end{align*}

Assume that $h^i:~K^i\rightarrow \underline{K}^{i-1}$ takes $(l,~m,~\underline{k})$ to $-\underline{k}$. We need to show that 
\begin{align*}
\alpha^i \circ \gamma_2^i-\beta^i \circ \gamma_1^i=d_{\underline{K}}^{i-1}\circ h^i +h^{i+1}\circ d^i_K
\end{align*}

for all $i\in \mathbb{Z}$ which shows that $\alpha \circ \gamma_2$ and $\beta \circ \gamma_1$ are homotopic maps in $K(\mathcal{A})$. This will show that they are same maps.\\

For all $(l,~m,~\underline{k})\in K^i$,
\begin{align*}
(\alpha^i \circ \gamma_2^i-\beta^i \circ \gamma_1^i)(l,~m,~\underline{k})=\alpha^i(\gamma^i_2(l,~m,~\underline{k}))-\beta^i(\gamma^i_1(l,~m,~\underline{k}))\\
=\alpha^i(l)-\beta^i(-m)=\alpha^i(l)+\beta^i(m)
\end{align*}

since $\mathcal{A}$ is additive. On the other hand,
\begin{align*}
(d^{i-1}_{\underline{K}}\circ h^i+h^{i+1}\circ d^i_K)(l, m, \underline{k})=d^{i-1}_{\underline{K}}(h^i(l,~m,~\underline{k}))+h^{i+1}(d^i_N(l,~m,~\underline{k}))\\
=d^{i-1}_{\underline{K}}(-k)+\alpha^i(l)+\beta^i(m)+d^{i-1}_{\underline{K}}(\underline{k})=\alpha^i(l)+\beta(m).
\end{align*}

This shows the maps are homotopic and the diagram is commutative.\\

We need to show that $\gamma_2$ is a quasi isomorphism. We have an exact triangle;
\begin{equation*}
\xymatrix{& L \ar[rd]\\ L[1]+MC(\beta) \ar[ur] & & MC(\beta) \ar[ll]}
\end{equation*}

This triangle is isomorphic to an exact triangle;
\begin{equation*}
\xymatrix{& MC(\beta) \ar[rd]\\ L[1]\ar[ur] & & L[1]+MC(\beta) \ar[ll]}
\end{equation*}

Then, we take its cohomology and the triangle still will be exact.
\begin{align*}
\xymatrix{& H^{\bullet}(MC(\beta)) \ar[rd]\\ H^{\bullet}(L[1]) \ar[ur] & & H^{\bullet}(L[1]+MC(\beta)) \ar[ll]}
\end{align*}\\

$H^{\bullet}(MC(\beta))=0$, so $H^{\bullet}(L[1])\cong H^{\bullet}(L[1]+MC(\beta))$ and 
\begin{align*}
L[1]+MC(\beta)=L[1]+M[1]+\underline{K}=K[1].
\end{align*}

Consequently, $H^{\bullet}(L[1])\cong H^{\bullet}(K[1])$. This means $H^{\bullet}(L)\cong H^{\bullet}(K)$, hence $\gamma_2$ is a quasi isomorphism.
\end{proof}

The pair $(f\circ f'',~g'\circ g'')$ is the composition of two pairs $(f,~g)$ and $(f',~g')$ as in the following commutative diagram;
\begin{equation*}
\xymatrix{ & & C''\ar[rd]^{g''} \ar[ld]_{f''} \\ & B'\ar[ld]_f \ar[rd]^g & & C'\ar[rd]^{g'} \ar[ld]_{f'}\\ A & & B & & C}
\end{equation*}

\end{document}